\documentclass{amsart}
\usepackage{amssymb}
\usepackage{latexsym}
\usepackage{amsmath}
\usepackage{euscript}
\usepackage{graphics}
\usepackage[all]{xy}
\usepackage{amsfonts,amsthm}
\usepackage{geometry}
\usepackage{MnSymbol} 
\usepackage{color}

\newcommand{\be}{\begin{equation}}
\newcommand{\ee}{\end{equation}}
\newcommand{\ba}{\begin{eqnarray}}
\newcommand{\ea}{\end{eqnarray}}
\newcommand{\baa}{\begin{eqnarray*}}
\newcommand{\eaa}{\end{eqnarray*}}
\newcommand{\bb}{\begin{thebibliography}}
\newcommand{\eb}{\end{thebibliography}}

\newcommand{\bi}[1]{\bibitem{#1}}
\newcommand{\lab}[1]{\label{#1}}
\newcommand{\re}[1]{(\ref{#1})}

\newcounter{my}
\newcommand{\he}%
   {\stepcounter{equation}\setcounter{my}%
   {\value{equation}}\setcounter{equation}0%
   }%
\newcommand{\she}%
   {\setcounter{equation}{\value{my}}%
    }%

\newcommand\ve{\varepsilon}

\newtheorem{theorem}{Theorem}[section]
\newtheorem{pr}[theorem]{Proposition}

\newtheorem{lemma}[theorem]{Lemma}

\theoremstyle{definition}

\numberwithin{equation}{section}

\begin{document}


\title[Isospectral deformations]{Persymmetric Jacobi matrices, isospectral deformations and orthogonal polynomials}


\author{Vincent Genest}
\author{Satoshi Tsujimoto}
\author{Luc Vinet}
\author{Alexei Zhedanov}

\address{Department of Mathematics, Massachusetts Institute of Technology, 77 Massachusetts Ave., Cambridge, MA 02139, USA}

\address{Department of Applied Mathematics and Physics, Graduate School of Informatics, Kyoto University, Yoshida-Honmachi, Sakyo-ku, Kyoto, Japan 606--8501}

\address{Centre de recherches math\'ematiques
Universite de Montr\'eal, P.O. Box 6128, Centre-ville Station,
Montr\'eal (Qu\'ebec), H3C 3J7}

\address{Institute for Physics and Technology\\
R.Luxemburg str. 72 \\
83114 Donetsk, Ukraine \\}

\begin{abstract}
Persymmetric Jacobi matrices are invariant under reflection with respect to the anti-diagonal. The associated orthogonal polynomials have distinctive properties that are discussed. They are found in particular to be also orthogonal on the restrictions either to the odd or to the even points of the complete orthogonality lattice. This is exploited to design very efficient inverse problem algorithms for the reconstruction of persymmetric Jacobi matrices from spectral points. Isospectral deformations of such matrices are also considered. Expressions for the associated polynomials and their weights are obtained in terms of the undeformed entities.
\end{abstract}



\maketitle

\section{Introduction}

\setcounter{equation}{0}

Jacobi matrices that are invariant under reflection with respect to the anti-diagonal are said to be persymmetric or mirror-symmetric. 
These matrices are known to possess tractable inverse spectral problems\cite{Atkinson,BG,BS} and  arise in a number of situations especially in engineering (see for instance \cite{Gladwell}). Interestingly, they have recently appeared in the analysis of perfect state transfer in quantum spin chains\cite{Albanese,Kay,VZ_PST}. In the latter context the main question is: what should the design of the chain be if it is to act as a quantum wire and transport quantum states from one 
location to another with probability one. The answer exploits the connection between orthogonal polynomials and Jacobi matrices as well as the features that mirror-symmetry brings. The present paper provides a systematic examination of the properties that the reflection invariance of persymmetric matrices induces on the associated orthogonal polynomials.

Isospectral deformations of persymmetric Jacobi matrices  have also found applications in the transport of quantum states along spin chains for instance.
Such Jacobi matrices have been seen to lead to spin chain models that exhibit the phenomenon of fractional revival \cite{Dai,GVZ_1506,GVZ_1507,Kay}.
When this occurs , small clones of the initial state are reproduced at the beginning and the end of the wire periodically.
Fractional revival also allows to transport quantum information and 
may serve as a mechanism to generate maximally entangled states.
This motivates the study of the properties of the orthogonal polynomials associated to isospectral deformations of persymmetric Jacobi matrices,
a task that is also carried out here.

The outline of the paper is thus as follows. In section 2, we review relevant aspects of the theory of Jacobi matrices and orthogonal polynomials and introduce persymmetric Jacobi matrices and their general features.
We obtain in Section 3 the essential properties of the persymmetric orthogonal polynomials. We show in particular  that these polynomials are orthogonal on the sublattices made out either of the odd or of the even points of the full orthogonality grid. 
This is exploited to present two new efficient algorithms for the reconstruction of  persymmetric Jacobi matrices from their spectra. 
Isospectral deformations are considered in Section 4 where the associated polynomials and weights are given in terms of those of the undeformed persymmetric Jacobi matrix.

\section{Persymmetric Jacobi matrices}
\setcounter{equation}{0}
Consider the tri-diagonal or Jacobi Hermitian matrix of size $N+1 \times N+1$
\[
J =
 \begin{pmatrix}
  b_{0} & a_1 & 0 &    \\
  a_{1} & b_{1} & a_2 & 0  \\
   0  &  a_2 & b_2 & a_3    \\
   &   &  \ddots &    \ddots  \\
      & & \dots &a_{N-1} & b_{N-1} & a_N \\   
            & & \dots &0 & a_N & b_N
         
\end{pmatrix}.
\] 
In what follows we will assume that all entries $a_i, b_i, \: i=0, 1, \dots$ are real. 

Let $e_i, \: i=0,1,\dots, N$ be the canonical basis in the $N+1$-dimensional linear space such 
\be
J e_n = a_{n+1} e_{n+1} + b_n e_n + a_n e_{n-1}, \quad n=0,1,\dots N \lab{J_e}. \ee
In order for relations (\ref{J_e}) to be valid for $n=0$ and $n=N$ we put additionally $a_0=a_{N+1}=0$. Of course, this condition is not a restriction but rather a convention.  

Introduce also the eigenvectors $X_s$ of the matrix $J$
\be
J X_s = x_s X_s, \quad s=0,1,\dots N \lab{JX} \ee
with eigenvalues $x_s$. It is well known that if $a_i \ne 0$ 
then all the eigenvalues of the matrix $J$ are real and nondegenerate \cite{Atkinson}, i.e. 
\be
x_s \ne x_{s'}, \quad \mbox{if} \quad s \ne s' \lab{ndeg_x}. \ee 
One can expand the vectors $X_s$ in terms of the basis vectors $e_n$ \cite{Atkinson}:
\be
X_s = \sum_{n=0}^N \sqrt{w_s} \chi_n(x_s) e_n, \lab{X_e} \ee 
where $\chi_n(x)$ are orthonormal polynomials defined by the initial conditions $\chi_{-1}(x)=0, \: \chi_0(x)=1$ and the 3-term recurrence relations 
\be
a_{n+1} \chi_n(x) + b_n \chi_n(x) + a_n \chi_{n-1}(x) = x \chi_n(x). \lab{3-term_chi} \ee
The reciprocal expansion of the vectors $e_n$ in terms of the eigenvectors is \cite{Atkinson}:
\be
e_n =\sum_{s=0}^N \sqrt{w_s} \chi_n(x_s) X_s. \lab{e_n_chi} \ee  
In view of the orthonormality of the two bases,
the polynomials $\chi_n(x)$ are orthogonal with respect to the weight function $w_s$:
\be 
\sum_{s=0}^N w_s \chi_n(x_s) \chi_m(x_s) = \delta_{nm}. \lab{osrt_chi} \ee

Sometimes it is convenient to introduce the so-called monic orthogonal polynomials 
\be
P_n(x) = a_1 a_2 \dots a_n \chi_{n}(x). \lab{P_monic} \ee
These polynomials have the expansion
\be
P_n(x) = x^n + O(x^{n-1}) \lab{P_expans} \ee
and satisfy the 3-term recurrence relation
\be
P_{n+1}(x) + b_n P_n(x) + u_n P_{n-1}(x) = x P_n(x), \lab{3_term_P} \ee
where $u_n = a_n^2>0$.
The recurrence relation \re{3_term_P} corresponds to the Jacobi matrix
\be
K =
 \begin{pmatrix}
  b_{0} & 1 & 0 &    \\
  u_{1} & b_{1} & 1 & 0  \\
   0  &  u_2 & b_2 & 1    \\
   &   &  \ddots &    \ddots  \\
      & & \dots &u_{N-1} & b_{N-1} & 1 \\   
            & & \dots &0 & u_N & b_N
         
\end{pmatrix}, 
\lab{matr_K} \ee
which is related to the Jacobi matrix $J$ by a similarity transformation
\be
K = S J S^{-1} \lab{K_J_S} \ee
with some diagonal matrix $S$. Sometimes the matrix $K$ is preferable because one can change the signs of the elements $a_n$ of the matrix $J$ without changing the spectral data $\{x_s, \: s=0,1,\dots, N\}$.

Note that the eigenvalues $x_s$ are the roots of the characteristic polynomial $P_{N+1}(x)$ \cite{Atkinson}:
\be
P_{N+1}(x) = (x-x_0) (x-x_1) \dots (x-x_N), \lab{P_N+1_roots} \ee
where the polynomial $P_{N+1}(x)$ can be determined from the recurrence relation \re{3_term_P} for $n=N$.

In what follows we will assume that the eigenvalues are 
ordered as follows from the smallest to the largest:
\be
x_0<x_1<x_2<\dots <x_N \lab{order_x}. \ee

The weights $w_s$ are given by the the formula \cite{Chi}
\be
w_s = \frac{h_N}{P_{N}(x_s) P_{N+1}'(x_s)}, \quad s=0,1,\dots,N,  \lab{w_s_expl} \ee
where 
\be
h_n = u_1 u_2 \dots u_n \lab{h_def} \ee
are normalization constants. 

For Hermitian Jacobi matrices with nonzero entries $a_n$, all weights are nonnegative, $w_s \ge 0$, and moreover they are normalized \cite{Chi}
\be
\sum_{s=0}^N w_s =1. \lab{norm_w} \ee
Note that for monic orthogonal polynomials the orthogonality relation looks as 
\be
\sum_{s=0}^N w_s P_n(x_s) P_m(x_s) = h_n \: \delta_{nm}.  \lab{ort_h} \ee 

There is an important interlacing property of the zeros of the polynomials of an orthogonal set
 \cite{Atkinson}: all zeros of the polynomial $P_n(x), \: n=1,2,\dots, N$ are distinct and moreover any zero of the polynomial $P_n(x)$ lies between two neighboring zeros of the polynomial $P_{n+1}(x)$. In particular, this means that the sequence $P_N(x_s), \: s=0,1,2,\dots$ has alternating signs:
\be
P_N(x_s) = (-1)^{N+s} W_s, \quad s=0,1,\dots, N, \lab{alt_P_N}   \ee
where $W_s$ is a strictly positive sequence $W_s>0$.

Let $R$ be the ``reflection'' matrix
\[
R=\begin{pmatrix}
  0 & 0 & \dots & 0 & 1    \\
  0 & 0 & \dots  & 1 & 0  \\
    \dots  & \dots & \dots & \dots & \dots      \\
   1 & 0 &  \dots & 0 &0  \\
\end{pmatrix}.
\]
Clearly, $R$ is an involution, i.e.
\be
R^2 = I, \lab{R^2} \ee
where $I$ is the identity matrix. Hence the eigenvalues of the matrix $R$ are either 1 or -1. 
The action of $R$ on the basis vectors $e_n$ is obviously:
\be
R e_n = e_{N-n} \lab{R_e_n}. \ee
The Hermitian matrix $J$ is called persymmetric (or mirror-symmetric) \cite{Gladwell} if it commutes with the reflection matrix
\be
JR = RJ \lab{JR}. \ee
Condition \re{JR} means that the entries of the persymmetric matrix satisfy the conditions
\be
a_{N+1-i}=a_i, \quad b_{N-i} = b_i, \quad i=0,1,\dots N. \lab{sym_ab} \ee
Hence a persymmetric matrix $J$ takes the form
\[
J =
 \begin{pmatrix}
  b_{0} & a_1 & 0 &    \\
  a_{1} & b_{1} & a_2 & 0  \\
   0  &  a_2 & b_2 & a_3    \\
   &   &  \ddots &    \ddots  \\
      & & \dots &a_2 & b_1 & a_1 \\   
            & & \dots &0 & a_1 & b_0
\end{pmatrix}.
\] 
For the Jacobi matrix $K$ \re{matr_K} the persymmetric property means that
\be
KR = R K^{T}, \lab{per_K} \ee  
where $K^{T}$ is the transposed matrix.

From the commutativity of the matrix $J$ with the matrix $R$ and from the nondegeneracy of the eigenvalues of matrix $J$, it follows that any eigenvector $X_s$ should be also an eigenvector of the reflection operator $R$:
\be
R X_s = \ve_s X_s, \lab{RX} \ee
where $\ve_s = \pm 1$. We will specify the values $\ve_s$ in the next section.

\section{Basic properties of persymmetric orthogonal polynomials}
\setcounter{equation}{0}
In this section we present the basic properties of the orthogonal polynomials $\chi(x)$ (or $P_n(x)$) 
that are associated to the persymmetric
Jacobi matrices $J$. Some of these properties are well known, some are new.

First of all, consider the property \re{RX}. Expanding the eigenvector $X_s$ in terms of the basis $e_n$ and using \re{X_e} we arrive at the relation
\be
\chi_{N-n}(x_s) = \ve_s \chi_n(x_s). \lab{chi_chi_s} \ee 
Relation \re{chi_chi_s} should be valid for all $n=0,1,\dots,N$. Choose $n=N$. Then $\chi_0=1$ and relation \re{chi_chi_s} becomes
\be
\chi_N(x_s) = \ve_s, \lab{chi_ve} \ee
where $\ve_s = \pm 1$. But by property \re{alt_P_N} we have that the only choice for $\ve_s$ is
\be
\ve_s = (-1)^{N+s}. \lab{ve_s} \ee
We thus have the following.
\begin{lemma}\label{lem_31}
The orthonormal polynomials corresponding to the persymmetric matrix $J$ satisfy the relation
\be
\chi_{N-n}(x_s) = (-1)^{N+s} \chi_n(x_s), \quad n=0,1,\dots, N. \lab{lem_zer_chi} \ee  
\end{lemma} 
For the monic polynomials we have a similar relation:
\be
P_{N-n}(x_s) = (-1)^{N+s} \nu_n P_n(x_s), \lab{lem_zer_P} \ee 
where
\be
\nu_n = \sqrt{h_{N-n}/h_n}. \lab{nu_hh} \ee
Taking into account that for persymmetric matrix  $u_{N+1-n}=u_n$, we have  $h_{N-n} = h_N/h_n$ and hence 
\be
\nu_n = \left\{
\begin{array}{cl}
\frac{h_{N/2}}{h_n}, &  N \quad \mbox{even},\\
\frac{\sqrt{u_{(N+1)/2}} \: h_{(N-1)/2}}{h_n}, & N \quad \mbox{odd}.
\end{array}
\right . 
\lab{nu_n} \ee
This lemma is well known . For other proofs see, e.g. \cite{BG}, \cite{Bor}.

Using Lemma \ref{lem_31}, we obtain from formula  \re{w_s_expl} the following expression for the weights
\be
w_s = (-1)^{N+s} \frac{h_N}{P_{N+1}'(x_s)}, \quad s=0,1,\dots,N. \lab{w_s_per} \ee 
This formula shows that orthogonality is defined uniquely from the spectrum $x_0,x_1,\dots, x_N$.

We thus have
\begin{lemma}\label{lem_32}
The spectral points $x_0, x_1, \dots, x_N$ determine the coefficients $b_0, b_1, \dots$ and $u_1, u_2, \dots $ of the  persymmetric Jacobi matrix $K$ uniquely. 
\end{lemma} 
This lemma is well known (see, e.g. \cite{Gladwell}).

Using lemmas \ref{lem_31} and \ref{lem_32}, one can propose simple algorithms for the inverse spectral problem, i.e. 
for the reconstruction of
the entries $b_0, b_1, \dots$ and $u_1, u_2, \dots $ of the persymmetric matrix $K$ if the spectrum $x_0, x_1, \dots, x_N$ is given. These algorithms are described in \cite{BG}, \cite{Gladwell}, \cite{VZ_PST}.

The first algorithm exploits the formula \re{w_s_per} for the weights $w_s$. Indeed, starting from the spectrum $x_s$ one can compute the non-normalized weights 
\be
w_s = \frac{1}{P_{N+1}'(x_s)} \lab{n_norm_w} \ee
and then reconstruct the orthogonal polynomials $P_1(x), P_2(x), \dots$ and the corresponding recurrence coefficients $b_n, u_n$using the standard Gram-Schmidt orthogonalization procedure (see \cite{BG}, \cite{Gladwell} for details).

Another algorithm exploits the property $\chi_N(x_s) = (-1)^{N+s}$ of the orthonormal polynomials corresponding to a persymmetric Jacobi matrix. The first step of this algorithm 
consists in the construction of the
monic polynomial $P_N(x)$ using the Lagrange interpolation formula. Indeed, the above formula allows to construct the polynomial $\chi_N(x)$ uniquely from its values $(-1)^{N+s}$ at the prescribed points $x_s$. This immediately gives the monic polynomial $P_N(x)$. We thus know explicitly two polynomials: $P_N(x)$  and $P_{N+1}(x) = (x-x_0) (x-x_1) \dots (x-x_N)$. Using the standard Euclidean algorithm we can obtain - step-by-step - all the other polynomials $P_{N-1}(x), P_{N-2}(x), \dots P_1(x)$ together with the corresponding recurrence coefficients $u_n, b_n$. This algorithm was applied in \cite{VZ_PST} in order to characterize spin chains which admit perfect state transfer.

We now present a new property of the persymmetric polynomials. We first introduce the following notations. Let $\Omega_0(x)$ and $\Omega_1(x)$ be the characteristic polynomials of the even and odd spectral points $x_s$. 
More precisely, let
\be
\Omega_0(x) = \left\{  { (x-x_0)(x-x_2) \dots (x-x_{N-1}), \quad N \quad \mbox{odd}, \atop (x-x_0)(x-x_2) \dots (x-x_{N}), \quad N \quad \mbox{even}, } \right . \lab{OM0} \ee 
and 
\be
\Omega_1(x) = \left\{  { (x-x_1)(x-x_3) \dots (x-x_{N}), \quad N \quad \mbox{odd}, \atop (x-x_1)(x-x_3) \dots (x-x_{N-1}), \quad N \quad \mbox{even}. } \right . \lab{OM1} \ee 
Note that for $N$ odd both polynomials $\Omega_0(x)$ and $\Omega_1(x)$ have degree $(N+1)/2$. For $N$ even, the polynomial $\Omega_0(x)$ has degree $N/2+1$ while polynomial $\Omega_1(x)$ has degree $N/2$. Recall that all the eigenvalues are assumed 
to be ordered in an increasing fashion
$x_0<x_1<x_2<\dots<x_N$ and that all the eigenvalues are distinct.

Introduce also the notation
\be
\sigma_0 = x_0 + x_2  \dots, \quad \sigma_1 = x_1 + x_3 + \dots, \lab{s_01} \ee 
where the summation is assumed over all even or all odd spectral points.

We have the following result.
\begin{lemma}
Let $N=2L+1$ be odd. The monic persymmetric orthogonal polynomials $P_{L+1}(x)$ and $P_L(x)$ 
are given by
\be
P_{L+1}(x) = \frac{\Omega_0(x) + \Omega_1(x) }{2}, \quad P_{L}(x) = \frac{\Omega_0(x) - \Omega_1(x) }{\sigma_1-\sigma_0} \lab{P_L_L+1_odd} \ee
and the recurrence coefficient $u_{L+1}$ has the following expression
\be
u_{L+1}=\frac{(\sigma_1-\sigma_0)^2}{4}. \lab{u_L+1_odd} \ee
Let $N=2L$ be even. In this case the monic polynomial $P_{L+1}(x)$ and $P_L(x)$ 
read
\be
P_{L+1}(x) = \frac{\Omega_0(x) +(x+\sigma_1-\sigma_0)\Omega_1(x)}{2}, \quad     P_L(x) =\Omega_1(x) \lab{P_L_even} \ee
and the recurrence coefficient $b_L$ is
\be
b_L=\sigma_0-\sigma_1. \lab{b_L_even} \ee
\end{lemma}
\begin{proof}
Consider first the case with $N$ odd. From formula \re{lem_zer_P} we have that the monic polynomial $P_{L+1}(x) -  \nu_L P_L(x)$ of degree $L+1$ has zeros at points $x_0, x_2, \dots, x_{N-1}$. Similarly, the monic polynomial $P_{L+1}(x) +  \nu_L P_L(x)$ of degree $L+1$ has zeros at points $x_1, x_3, \dots, x_N$. This means that
\be
P_{L+1}(x) -  \nu_L P_L(x) = \Omega_0(x), \quad P_{L+1}(x) +  \nu_L P_L(x) = \Omega_1(x). \lab{PP_OmOm} \ee
From \re{PP_OmOm} we obtain
\be
P_{L+1}(x) =\frac{\Omega_0(x) + \Omega_1(x) }{2}, \quad P_L(x)=\frac{\Omega_1(x) - \Omega_0(x) }{2 \nu_L}. \lab{P_LL_1} \ee
From \re{nu_hh} we have $\nu_L= \sqrt{u_{L+1}}>0$. On the other hand we have
$$
\Omega_1(x) - \Omega_0(x) = (\sigma_0 - \sigma_1) x^L + O(x^{L-1})
$$
and hence from the second formula \re{P_LL_1} it follows that $2 \nu_L=\sigma_1-\sigma_0$. This gives the formula \re{u_L+1_odd} and the second formula \re{P_L_L+1_odd}.

Let $N=2L$ be even. Putting $n=L$ in \re{lem_zer_P}, we immediately obtain the second formula \re{P_L_even}. Put then $n=L-1$. Formula \re{lem_zer_P} gives the relations
\be
P_{L+1}(x) - u_L P_{L-1}(x) = \Omega_0(x), \quad P_{L+1}(x) + u_L P_{L-1}(x) = \Omega_1(x) (x-\beta), \lab{PP_OmOm_2} \ee
where $\beta$ is some parameter (an additional factor $x-\beta$ is needed in the second formula \re{PP_OmOm_2} because the degree of the polynomial $\Omega_1(x)$ is $L$). From the recurrence relation \re{3_term_P} for the polynomials $P_n(x)$ we can rewrite the second formula \re{PP_OmOm_2} as
$$
(x-b_L) P_L(x) =(x-b_L) \Omega_1(x) = (x-\beta) \Omega_1(x) 
$$
whence $\beta = b_L$. From the same recurrence relation \re{3_term_P}, the first formula \re{PP_OmOm_2} becomes
\be
2 P_{L+1}(x) = \Omega_0(x) -(x-b_L) \Omega_1(x). 
\lab{P_om_5} \ee
This gives formula \re{P_L_even}. The expression \re{b_L_even} for $b_L$ follows easily from \re{P_om_5} by comparing the coefficients in front of $x^L$.
\end{proof}

The next property of the persymmetric polynomials seems to be new also. This property specifies the orthogonality relations for these polynomials.

Let $P_n(x), \: n=0,1,\dots, N$ be monic persymmetric polynomials. They are orthogonal with respect to the weights $w_s$ given by formula \re{w_s_per}. It turns out that the polynomials $P_n(x)$ for $n<N/2$ have orthogonality relations that are more simple.
In order to see this we first need to 
introduce finite differences and a few of their properties.
Let  $x_0, x_1, \ldots$ be a prescribed set of distinct points. 
Consider the characteristic polynomial $P_{N+1}(x) = (x-x_0)(x-x_1)...(x-x_N)$ of the first $N +1$ points. 
The finite difference $f[x_0,x_1,\ldots,x_N]$ of $N$-th order of the function $f(x)$ may be defined as \cite{NSU}:
\begin{eqnarray}
 f[x_0,x_1,\ldots,x_N] = \sum_{s=0}^{N}\dfrac{f(x_s)}{P_{N+1}'(x_s)}.
\end{eqnarray}
In the simplest case of the first order finite difference
\begin{eqnarray}
 f[x_0,x_1] = \dfrac{f(x_0)-f(x_1)}{x_0-x_1}.
\end{eqnarray}
Among the many useful properties of finite differences we need the following: for any polynomial $\pi(x)$ of degree $< N$, the divided difference of $N$-th order is zero \cite{NSU}. This property is the finite difference analogue of the well known property that the $N$-th order derivative of a polynomial of degree $< N$ is zero.

We already know that the persymmetric polynomials are orthogonal with respect to the weight function
\begin{eqnarray*}
 w_s = (-1)^{N+s}\dfrac{h_N}{P_{N+1}(x_s)}, \quad s=0,1,\ldots,N.
\end{eqnarray*}
In particular, we can introduce the corresponding moments
\begin{eqnarray}
 c_n = \sum_{s=0}^{N}x_s^n w_s = \sum_{s=0}^{N}(-1)^{N+s}\dfrac{x_s^n}{P_{N+1}'(x_s)}, \quad n= 0,1,\ldots,N.
 \label{moment:cn}
\end{eqnarray}
Using these moments we can obtain the following expression for the orthogonal polynomials \cite{Chi}
\begin{eqnarray*}
 P_n(x)=
\frac{1}{\Delta_n}
\left(\begin{array}{cccc}
 c_0&c_1 &\cdots &c_{n-1} \\
 c_1&c_2 &\cdots &c_{n} \\
 \cdots&\cdots &\cdots &\cdots \\
 c_{n-1}&c_{n} &\cdots &c_{2n-2} \\
 1&x &\cdots &x^n \\
      \end{array}\right) ,
\end{eqnarray*}
where $\Delta$ is the Hankel determinant
\begin{eqnarray*}
\left(\begin{array}{cccc}
 c_0&c_1 &\cdots &c_{n-1} \\
 c_1&c_2 &\cdots &c_{n} \\
 \cdots&\cdots &\cdots &\cdots \\
 c_{n-1}&c_{n} &\cdots &c_{2n-2} \\
      \end{array}\right) .
\end{eqnarray*}
It is seen that in order to obtain the expression of the polynomial $P_n(x)$, we need the moments $c_0, c_1, \ldots, c_{2n-1}.$ Note that $c_0 = 1$ given the normalization condition (\ref{norm_w}).

Assume that $N = 2L + 1$ is even. 
From (\ref{moment:cn}) we then have on the one hand
\begin{eqnarray}
 c_n= \sum_{s=0}^{L}\dfrac{x_{2s+1}^n}{P_{N+1}'(x_{2s+1})}
    - \sum_{s=0}^{L}\dfrac{x_{2s}^n}{P_{N+1}'(x_{2s})}.
\label{eq326}
\end{eqnarray}
On the other hand 
using the property of finite differences that we singled out, we have
\begin{eqnarray}
 \sum_{s=0}^{N}\dfrac{x_{s}^n}{P_{N+1}'(x_{s})}
 =
 \sum_{s=0}^{L}\dfrac{x_{2s}^n}{P_{N+1}'(x_{2s})}
    + \sum_{s=0}^{L}\dfrac{x_{2s+1}^n}{P_{N+1}'(x_{2s+1})}
 =0 , \quad n<N.
\label{eq327}
\end{eqnarray}
Combining (\ref{eq327}) and (\ref{eq326}) and (\ref{moment:cn}) we obtain that for $n < N$
\begin{eqnarray}
 c_n= 2 \sum_{s=0}^{L}x_{2s+1}^n w_{2s+1}
    = 2 \sum_{s=0}^{L}x_{2s}^n w_{2s}.
\label{eq326a}
\end{eqnarray}
This means that the moments $c_1,c_2,...,c_{N-1}$ can be calculated using only half of the spectral points: 
that is using either the odd or the even spectral points.

We thus arrive at the following result.
\begin{lemma} \label{lem34}
 Let $N = 2L+1$ be odd. The polynomials $P_n(x)$ for $n = 0, 1, \ldots, L$ are orthogonal either with respect to the weights $w_0, w_2, \ldots, w_{2L}$ on the grid $x_0, x_2, \ldots, x_{2L}$ or with respect to the weights $w_1, w_3, \ldots, w_{2L+1}$ on the grid $x_1, x_3, \ldots, x_{2L+1}$. 
If $N = 2L$ is even, then the polynomials $P_n(x)$ for $n = 0, 1, \ldots, L - 1$ are orthogonal with respect to the weights $w_1, w_3, \ldots, w_{2L-1}$ on the grid $x_1, x_3, \ldots, x_{2L-1}$.
\end{lemma}

With the help of this lemma, we can propose a most efficient
algorithm for solving inverse spectral problem for persymmetric Jacobi matrices with prescribed spectra.

Recall that a persymmetric matrix (equivalently, the recurrence coefficients $b_n$,$u_n$) can be reconstructed uniquely from the spectrum $x_0, x_1, \ldots, x_N$ \cite{BG}, \cite{Gladwell}. 
As already mentioned,
one approach is to use the Gram-Schmidt orthogonalization procedure with respect to the prescribed weights (\ref{n_norm_w}), another one is to call upon the 
Euclidean algorithm which involves the division
of the polynomial $P_{n+1}(x)$ by the polynomial $P_{n}(x)$ to get the next polynomial $P_{n-1}(x)$. 
In the last case we start from the given polynomial $P_{N+1}(x) = (x-x_0)(x-x_1)\cdots(x-x_N)$ and the polynomial $P_{N}(x)$ which can be reconstructed by the Lagrange interpolation formula using its characteristic property $\chi_N (x_s) = (-1)^{N+s}$. 
We can then iteratively obtain all the remaining polynomials $P_n(x), \: n=N-1,N-2,\dots, 1$ and the corresponding recurrence coefficients $b_n, u_n$. 
This algorithm thus allows to reconstruct 
efficiently persymmetric Jacobi matrices from prescribed spectral data $x_0, x_1,\dots, x_N$. 
It can be applied in particular to the problem of finding how the couplings of spin chains should be pre-engineered so as to ensure the occurrence of perfect state transfer along the chain \cite{VZ_PST}.

One can however drastically improve the efficiency of the inverse problem algorithm by making use of the special properties of the persymmetric polynomials that we have identified.

Indeed, consider, e.g. the case where $N = 2L+1$ is odd. We already know explicitly the recurrence coefficient $u_{L+1}$ and the polynomials $P_{L+1}(x)$ and $P_{L}(x)$ from formulas (\ref{P_L_L+1_odd}). 
From this point on, using the Euclidean algorithm we can 
reconstruct the polynomials
$P_{L-1}(x), P_{L-2}(x), \ldots, P_1(x)$ and the corresponding recurrence coefficients 
$b_L,b_{L-1},\ldots, b_0, u_L, u_{L-1},\ldots, u_1$.

The efficiency of this method 
lies in the fact that start roughly from the the middle of the polynomial sequence $P_N(x), P_{N-1}(x), \dots, P_1(x), P_0(x)$ 
instead of starting from its end (i.e. from the polynomials $P_{N+1}(x), P_N(x)$). 
Moreover, the polynomials $P_{L+1}(x)$ and $P_L(x)$ are given by the simple explicit  expressions \re{P_L_L+1_odd} 
while otherwise, as indicated before, we need first to calculate the polynomial $P_N(x)$ using the Lagrange interpolation method. This 
indicates that the proposed new 
algorithm will be quicker than the previous one.

Another efficient algorithm 
can be obtained if one exploits the orthogonality of the persymmetric polynomials on the even sublattice $x_0, x_2, \ldots, x_{2L}$ 
described in Lemma \ref{lem34}.
We can then apply the standard Gram-Schmidt 
orthogonalization procedure to reconstruct the polynomials $P_1(x), P_2(x), \dots, P_L(x)$, i.e. only 
half of the set of all the polynomials $P_1, P_2, \dots, P_N(x)$. 
However, for the purpose of the reconstruction of the persymmetric Jacobi matrix this information is enough. Moreover, the Gram-Schmidt orthogonalization algorithm can be stopped at the level of the polynomial $P_{L-1}(x)$ because 
the next two polynomials $P_L(x)$ and $P_{L+1}(x)$ are already known explicitly  \re{P_L_L+1_odd}. 

We thus see that these two new algorithms are poised to be
much more efficient and faster than the standard ones.

For the case of even $N$ one can design similar algorithms.

\section{Isospectral deformation of persymmetric matrices}
\setcounter{equation}{0}
Let $V$ be the $(N+1) \times (N+1)$ matrix defined 
in terms of the parameter $\theta$,  $0 < \theta <  2 \pi$.
For $N$ odd, take
\begin{eqnarray}
 V = \left(
\begin{array}{cccccc}
 \sin \theta& & & & & \cos\theta\\
&\ddots & & &\udots&\\
& &\sin \theta & \cos\theta&\\
& &\cos \theta & -\sin\theta&\\
&\udots & & &\ddots&\\
 \cos \theta& & & & & -\sin\theta\\
\end{array}
\right)
\end{eqnarray}
and for $N$ even, write
\begin{eqnarray}
 V = \left(
\begin{array}{ccccccc}
 \sin \theta& && & && \cos\theta\\
&\ddots && &&\udots&\\
& &\sin \theta &0 & \cos\theta&&\\
& & 0 & 1 & 0 &&\\
& &\cos \theta &0 & -\sin\theta&&\\
&\udots & &&& \ddots&\\
 \cos \theta& & & & && -\sin\theta\\
\end{array}
\right).
\end{eqnarray}
Clearly the matrix matrix $V$ is symmetric $V$ = $V^{T}$ and is moreover an involution, i.e.
\begin{eqnarray}
 V^2 =I,
\end{eqnarray}
where $I$ is the identity matrix.

Consider the following isospectral transformation of the persymmetric Jacobi
matrix (\ref{per_K})
\begin{eqnarray}
 \tilde J= V J V.
\end{eqnarray}
Obviously, this transformation is isospectral, i.e. the matrix $\tilde J$ has the same spectrum $\{\lambda_0, \lambda_1, \ldots, \lambda_N\}$ as the matrix $J$. Moreover, it is easily verified that the matrix $\tilde J$ 
is no longer persymmetric but remains
symmetric tridiagonal.
Furthermore the deviation of the matrix $\tilde J$ from the matrix $J$ is minimal.
Indeed, the only entries of $\tilde J$ that differ from those of $J$ are
\begin{eqnarray}
 \tilde a_{\frac{N+1}{2}} = a_{\frac{N+1}{2}} \cos 2\theta, \quad
 \tilde b_{\frac{N \mp 1}{2}} = b_{\frac{N-1}{2}}  \pm a_{\frac{N+1}{2}} \sin 2\theta,
\end{eqnarray}
for $N$ odd and
\begin{eqnarray}
 \tilde a_{\frac{N}{2}} = a_{\frac{N}{2}} (\cos \theta + \sin\theta), \quad
 \tilde a_{\frac{N}{2}+1} = a_{\frac{N}{2}} (\cos \theta - \sin \theta)
\end{eqnarray}
for $N$ even.

Because the matrix $\tilde J$ is tridiagonal, it has eigenvectors $\tilde X_s$ such that
\begin{eqnarray}
 \tilde X_s = \sum_{n=0}^{N}\sqrt{\tilde w_s}\tilde \chi_n(x_s)\, e_n,
\label{eigenvec:tilde Xs}
\end{eqnarray}
where $\tilde \chi_n(x)$ are orthonormal polynomials satisfying the recurrence relation
\begin{eqnarray}
 \tilde a_{n+1} \tilde \chi_{n+1}(x) + \tilde b_n \tilde \chi_n(x) + \tilde a_n \tilde \chi_{n-1}(x) = x \tilde \chi_n(x)
\end{eqnarray}
with the standard initial conditions $\tilde \chi_0=1, \tilde \chi_{-1}=0$. These polynomials are orthogonal with respect to the weights $\tilde w_s$:
\begin{eqnarray}
 \sum_{s=0}^n \tilde  w_s \tilde \chi_n (x_s) \tilde \chi_m(x_s) = \delta_{nm}.
\end{eqnarray}
The spectral points $x_s$ remain the same 
as those of the polynomials
$\chi_n(x)$ because the
Jacobi matrix $\tilde J$ has the same spectrum as the Jacobi matrix $J$.

Let $N$ be odd. Then we have
\begin{pr}
 For $N$ odd, the weights $\tilde w_s$ are
\begin{eqnarray}
 \tilde w_{2s} = w_{2s} (1-\sin 2 \theta),\,\, \tilde w_{2s+1} = w_{2s+1} (1+\sin 2\theta), \quad s = 0,1,\ldots,\frac{N-1}{2}.
\label{prop:tilde w_{2s}}
\end{eqnarray}
\end{pr}
\begin{proof}
 We have by definition
\begin{eqnarray}
 \tilde X_s = V X_s = \sum_{n=0}^{N}\sqrt{w_s} \chi_n(x_s) V e_n
  = \sum_{i,n} \sqrt{w_s} \chi_n(x_s) V_{in} e_i
\label{tilde X_s=VXs}
\end{eqnarray}
where $V_{in}$ are the entries of the matrix $V$. Comparing formulas (\ref{eigenvec:tilde Xs}) and (\ref{tilde X_s=VXs}) we arrive at the relations
\begin{eqnarray}
 \sqrt{\tilde w_s} \tilde \chi_n (x_s) = \sum_i \sqrt{w_s} V_{ni} \chi_i (x_s), \quad n,i= 0,1,\ldots,N.
\label{rel:tilde chi and chi}
\end{eqnarray}
Put $n=0$ in (\ref{rel:tilde chi and chi}):
\begin{eqnarray}
 \sqrt{\tilde w_s} = \sqrt{w_s}(\sin \theta\, \chi_0(x_s) + \cos \theta \,\chi_N(x_s))
  = \sqrt{w_s} (\sin \theta - (-1)^s\cos \theta).
\label{sqrt tilde chi}
\end{eqnarray}
Taking the square of both sides of (\ref{sqrt tilde chi}) we arrive at (\ref{prop:tilde w_{2s}}).
\end{proof}

\begin{pr}
 For $N$ odd the polynomials $\tilde \chi_n(x)$ are expressed in terms of the polynomials $\chi_n(x)$ as follows:
\begin{eqnarray}
 \tilde \chi_{n}(x) = 
\left\{
\begin{array}{cl}
 \chi_n(x),& \text{if }\,\, n \le \frac{N-1}{2},\\[1mm]
 \frac{1}{\cos 2\theta} \chi_n(x)
  - \frac{\sin 2\theta}{\cos 2\theta} \chi_{N-n}(x), & 
 \text{if }\,\, n > \frac{N-1}{2}.\\
\end{array}
\right.
\label{prop:tilde chi_{n}}
\end{eqnarray}
\end{pr}
\begin{proof}
From (\ref{rel:tilde chi and chi}) and (\ref{sqrt tilde chi}), we have the relations 
\begin{eqnarray}
 (\sin\theta -(-1)^s \cos\theta) \tilde \chi_n(x_s)
  = \sum_{i} V_{ni}\chi_i(x_s).
\label{prop odd N:proof:rel1}
\end{eqnarray}
When $n \le \frac{N-1}{2}$ relation (\ref{prop odd N:proof:rel1}) becomes
\begin{eqnarray}
 (\sin\theta -(-1)^s \cos\theta) \tilde \chi_n(x_s)
  = \sin\theta\chi_n(x_s) + \cos\theta \chi_{N-n}(x_s).
\label{prop odd N:proof:rel2}
\end{eqnarray}
Taking into account relation (\ref{lem_zer_chi}), we obtain from (\ref{prop odd N:proof:rel2})
\begin{eqnarray}
 \tilde \chi_n(x_s) = \chi_n(x_s), \quad s=0,1,\ldots,N, \quad n=0,1, \ldots,\frac{N-1}{2}.
\label{prop odd N:proof:rel3}
\end{eqnarray}
From (\ref{prop odd N:proof:rel3}) it follows trivially that
\begin{eqnarray}
 \tilde \chi_n(x_s) = \chi_n(x_s), \quad n= 0,1, \ldots, \frac{N-1}{2}.
\end{eqnarray}
When $n > \frac{N-1}{2}$ we have similarly
\begin{eqnarray}
  (\sin\theta -(-1)^s \cos\theta) \tilde \chi_n(x_s)
 = - (\sin\theta +(-1)^s \cos\theta) \chi_n(x_s).
\end{eqnarray}
It is easily verified that the polynomials
\begin{eqnarray}
 \pi_n(x) = \frac{1}{\cos 2\theta} \chi_n(x) 
 -\frac{\sin2\theta}{\cos 2\theta} \chi_{N-n}(x)
\end{eqnarray}
satisfy the same relations
\begin{eqnarray}
   (\sin\theta -(-1)^s \cos\theta) \pi_n(x_s)
   = - (\sin\theta +(-1)^s \cos\theta) \chi_n(x_s)
\end{eqnarray}
for all $s=0,1,\ldots,N$. 
Hence the polynomials $\pi_n(x)$ and $\tilde \chi_n(x)$ should coincide identically for $n=\frac{N+1}{2}, \frac{N+3}{2}, \ldots, N$ and we arrive formula (\ref{prop:tilde chi_{n}}).
\end{proof}
The isospectral perturbation of the persymmetric matrices described in this section has  an important application in 
the contruction of spin chains that  exhibit the fractional revival phenomenon, an effect that can be used for quantum information purposes to communicate data and generate entanglement \cite{GVZ_1506}, \cite{GVZ_1507}.
Moreover, the same perturbation of the persymmetric Jacobi matrix appears in the construction of a new finite-dimensional truncation of the Wilson polynomials leading to a new class of orthogonal polynomials, namely the para-Racah polynomials \cite{LVZ}.

\section{Conclusion}
In this paper we 
have described the properties of persymmetric matrices and of the corresponding orthogonal polynomials. 
This has allowed to offer a fast and efficient algorithm for reconstructing the persymmetric matrix from spectral data. 
We have shown moreover that there exists a simple 
isospectral deformation 
of persymmetric matrices that leads to explicit modified polynomials of interest.
We wish to stress once more that the results obtained in this paper are of relevance in studies of spin chains aimed at quantum information issues.

The properties of the persymmetric matrices reported here have further applications, e.g. in the theory of classical orthogonal polynomials on discrete lattices and in the theory of the Pad\'e interpolation. 
These questions will be considered elsewhere.

\bigskip\bigskip
{\Large\bf Acknowledgments}
\bigskip
\\
The authors wish to express their appreciation for the hospitality that
 the Centre de Recherches Math\'ematiques Universite de Montreal and the Graduate School of Informatics, Kyoto University have extended to some of them in the course of this work.

The research of LV is supported in part by a grant from the Natural Sciences and Engineering Research Council (NSERC) of Canada and 
the one of ST is supported by JSPS KAKENHI Grant Numbers 25400110.

\bb{99}

\bi{Albanese} C.~Albanese, M.~Christandl, N.~Datta and A.~Ekert, {\it Mirror inversion of quantum states in linear registers}, Phys. Rev. Lett. {\bf 93} (2004), 230502.

\bi{Atkinson} F.~V.~Atkinson, {\it Discrete and Continuous Boundary problems}, Academic Press, NY, London, 1964.

\bi{BG} C.~de~Boor and G.~H.~Golub {\it The numerically stable reconstruction of a Jacobi matrix from spectral data}, Lin. Alg. Appl. {\bf 21} (1978), 245--260.

\bi{BS} C.~de~Boor and E.~Saff, {\it Finite sequences of orthogonal polynomials connected by a Jacobi matrix}, Lin. Alg. Appl. {\bf 75} (1986), 43--55

\bi{Bor} A.~Borodin, {\it Duality of Orthogonal Polynomials on a
Finite Set}, J. Stat. Phys. {\bf 109}, (2002), 1109--1120.

\bi{Chi} T.~Chihara, {\it An Introduction to Orthogonal
Polynomials}, Gordon and Breach, NY, 1978.

\bi{Dai} L.~Dai, Y.~P.~Feng and L.~C.~Kwek, {\it Engineering quantum cloning through maximal entanglement between boundary qubits in an open spin chain}, J. Phys. A: Math. Theor., {\bf 43} (2010), 035302.

\bi{GVZ_1506} V.~X.~Genest, L.~Vinet and A.~Zhedanov,
{\it Exact fractional revival in spin chains},
arXiv:1506.08434, 2015.

\bi{GVZ_1507} V.~X.~Genest, L.~Vinet and A.~Zhedanov,
{\it Quantum spin chains with fractional revival}, 
arXiv:1507.05919, 2015.

\bi{Gladwell}  G.~M.~L.~Gladwell , {\it Inverse Problems in Vibration}, 400 pp., Martinus Noordhoff, 1986.

\bi{Kay} A.~Kay, {\it A Review of Perfect State Transfer and its
Application as a Constructive Tool}, Int. J. Quantum Inf. {\bf 8}
(2010), 641--676;  arXiv:0903.4274.

\bi{KLS} R.~Koekoek, P.~Lesky and R.~Swarttouw, {\it Hypergeometric Orthogonal Polynomials and Their Q-analogues}, Springer-Verlag, 2010.

\bi{LVZ} J.-M.~Lemay, L.~Vinet and  A.~Zhedanov {\it The para-Racah polynomials}, arXiv:1511.05215.

\bi{NSU} A.~F.~Nikiforov, S.~K.~Suslov and V.~B.~Uvarov, 
{\it Classical Orthogonal Polynomials of a Discrete Variable},
Springer, Berlin, 1991.

\bi{Sz} G.~Szeg\H{o}, 
{\it Orthogonal Polynomials}, fourth edition,  AMS, 1975.

\bi{VZ1} L.~Vinet and A.~Zhedanov, 
{\it A characterization of classical and semiclassical orthogonal polynomials from their dual polynomials},  
J. Comp. Appl. Math. {\bf 172} (2004), 41--48.

\bi{VZ_PST} L.~Vinet and A.~Zhedanov, 
{\it How to construct spin chains with perfect state transfer}, 
Phys. Rev. A {\bf 85} (2012), 012323.

\eb

\end{document}